\documentclass[12pt]{article}
\usepackage{amsmath, amsthm, amssymb}
\usepackage{array,color,colortbl} 
\usepackage[english]{babel}
\usepackage{graphicx}
\usepackage{epsf}
\usepackage{colordvi}
\usepackage{colortbl}
\usepackage{enumitem}

\theoremstyle{plain}
\newtheorem{theorem}{Theorem}
\newtheorem{lemma}[theorem]{Lemma}

\newtheorem{corollary}[theorem]{Corollary}

\def\Cref#1{Construction~$\ref{#1}$}
\def\cref#1{Corollary~$\ref{#1}$}

\renewcommand{\geq}{\geqslant}
\renewcommand{\leq}{\leqslant}

\def\dfrac#1#2{\lower0.15ex\hbox{\large$\frac{#1}{#2}$}} 

\title{Lower bounds on the sizes of defining sets in full $n$-Latin squares and full designs}
\author{Nicholas J. Cavenagh 
 \\
Department of Mathematics, \\
The University of Waikato, \\
Private Bag 3105, \\
Hamilton 3240, New Zealand \\
\texttt{nickc@waikato.ac.nz }
}

\begin{document}

\date{}
\maketitle

\begin{abstract}
The full $n$-Latin square is the $n\times n$ array with symbols $1,2,\dots ,n$ in each cell. 
In this paper we show, as part of a more general result, that any defining set for the full $n$-Latin square has size $n^3(1-o(1))$. The full design $N(v,k)$ is the unique simple design with parameters $(v,k,{v-2 \choose k-2})$; that is, the design consisting of all subsets of size $k$ from a set of size $v$. We show that any defining set for the full design $N(v,k)$ has size ${v\choose k}(1-o(1))$ (as $v-k$ becomes large). 
These results improve existing results and are asymptotically optimal. In particular, the latter result solves an open problem given in (Donovan, Lefevre, et al, 2009), in which it is conjectured that the proportion of blocks in the complement of a full design will asymptotically approach zero. 
\end{abstract}

\section{Introduction}
For convenience, we adopt the notation $N(a)$ for the set of positive integers $\{1, 2, . . . , a\}$.

We prove the key result on full $n$-Latin squares in terms of a more general combinatorial array. 
A (partial) $(m,n,t)$-balanced rectangle is an $m\times n$ array of multisets 
of size (at most) $t$ such that each element of $N(t)$ occurs (at most) $n$ times in each row and (at most) $m$ times per column.  
An $(m,n,t)$-balanced rectangle may be easily constructed by placing $N(t)$ in every cell; we call this the {\em full} $(m,n,t)$-balanced rectangle or $F_{m,n,t}$. 

A (partial) $n$-Latin square is a (partial) $(n,n,n)$-balanced rectangle and the {\em full} $n$-Latin square is the $n$-Latin square with $N(n)$ in every cell.   

A defining set $D$ of an $(m,n,t)$-balanced Latin rectangle $L$ is a partial $(n,m,t)$-balanced Latin rectangle with unique completion to $L$.  
In the example below, the partial $(2,3,3)$-balanced Latin rectangle 
(on the left) is not a defining set for $F_{2,3,3}$ (in the centre) as it also completes to another $(2,3,3)$-balanced 
Latin rectangle (on the right). 

\[ 
 \begin{tabular}{| l | l | l |}\hline
  1&  &   \\\hline
  & 2 & 2,3   \\\hline
 \end{tabular}
\quad
 \begin{tabular}{| l | l | l |}\hline
  1,2,3& 1,2,3 &1,2,3   \\\hline
1,2,3  & 1,2,3 & 1,2,3   \\\hline
 \end{tabular}
\quad
 \begin{tabular}{| l | l | l |}\hline
  1,1,2& 2,3,3 &1,2,3   \\\hline
2,3,3  & 1,1,3 & 1,2,3   \\\hline
 \end{tabular}
\]

The {\it size} or the number of entries in a partial $(m,n,t)$-Latin rectangle $L$, denoted by $|L|$, is the cardinality of the multiset sum of the multisets in each cell of $L$ (i.e. the sum of multiplicities of each element over all the cells). 
To clarify, the above structures have sizes $4$, $12$ and $12$, respectively.

We identify any partial $n$-Latin square as {\it saturated} if each cell is either empty or contains $N(n)$. 
 Otherwise it is  {\it non-saturated}. 
In \cite{CRa}, it was shown that  a saturated critical set for the full $n$-Latin square has size exactly equal to $n^3-2n^2-n$ and that any defining set has size at least $(n^3-2n^2+2n)/2.$
We significantly improve this lower bound in Corollary \ref{bbound}, showing in fact that the complement of a defining set for $F_{n,n,n}$ asymptotically tends to zero in proportion to $|F_{n,n,n}|=n^3$.  
The bound presented is unlikely to be exact; in \cite{CRa} defining sets of size $(n-1)^3+1$ of the full $n$-Latin square are constructed for $n\geq 2$; this remains the smallest known construction. 
The structure of minimal defining sets of $F_{m,n,2}$ is completely determined in \cite{CR2}.

The study of critical sets in full $n$-Latin squares has the potential to yield information on critical sets in Latin squares; for example, intersecting a defining set of $F_{n,n,n}$ with any Latin square $L$ of order $n$ results in a defining set of $L$. See \cite{Cav,Ke} for surveys on defining sets in Latin squares. 

The idea of a full $n$-Latin square was motivated by the analogous concept of a full design (see \cite{AMM, DY, DLWY, 23, GS2, GSS, KY, LW}). 
A $(v,k,\lambda)$-design is a collection of subsets of $N(v)$ (which are called {\em blocks}), each of size $k$, 
such that each (unordered) pair from $N(v)$ occurs in exactly $\lambda$ blocks. 
A design is {\em simple} if there are no repeated blocks. 
There is a unique simple $(v,k,{v-2\choose k-2})$-design consisting of all the possible  
subsets of size $k$ from the foundation set $N(v)$; we call this the {\em full} design $F(v,k)$. 

A {\em defining set} for a $(v,k,\lambda)$ design $D$ is a subset of $D$ which is not the subset of another design with the same parameters (where repeated blocks are allowed). See \cite{GS} for a survey of known results on defining sets in designs. 
Analogously to full Latin squares, the intersection of the defining set of the full design $F(v,k)$ with any simple design $D$ with parameters $v$ and $k$ gives a defining set for $D$.  

The previous best known lower bound on the size of a defining set in the full design (for all $v\geq 3$) is $F(v,3)$ is $3{v\choose 3}/7$ (\cite{23}). 
In the same paper it is conjectured in the conclusion that: ``\dots the proportion of blocks in the complement will asymptotically approach zero.'' 
We prove this conjecture for any full design $F(v,k)$, where $v-k$ tends to infinity, in Theorem \ref{abigone}. 
 The smallest known size of a constructed defining set for $F(v,3)$ is $(v^3-6v^2+5v+6)/6$ (\cite{23}).  In \cite{GS3} defining sets for $F(v,k)$ of size 
 ${v\choose k}-(v^2+3v-2vk+2k^2-8)/2$ are given whenever $v\geq k+2\geq 5$.  
In \cite{LW} two constructions of minimal defining sets of $F(v,k)$ are given, of size ${v \choose k}-(v^2-v-k^2+k+2)/2$ (where $v\geq k+3$) and of size 
${v \choose k}+(k-1)(k+2)/2-kv$ (where $v\geq k+2$).

Although this paper gives asymptotically optimal solutions for the sizes of the smallest defining sets in both full designs and full $n$-Latin squares, it remains an open problem to determine exact bounds. 
The simplicity of constructions for the smallest defining sets known so far suggests that elegant formulae may exist for these bounds.


\section{Defining sets of full $(m,n,t)$-Latin rectangles}

In the following, $D$ is a simple, partial $(m,n,t)$-Latin rectangle. That is, 
$D$ always has a completion to the full $(m,n,t)$-Latin rectangle $F_{m,n,t}$.    
For each pair $\{a,b\}\subset N(t)$,
we define $S_{a,b}(D)$ (or $S_{a,b}$ if the context is clear) to be the set of cells in $D$ which contain {\em neither} $a$ nor $b$. 

\begin{lemma}
Let $t\geq 2$. 
Suppose that $D$ is a defining set for $F_{m,n,t}$. For each pair $\{a,b\}\subset N(n)$, $|S_{a,b}(D)|\leq m+n-1$. 
\end{lemma}

\begin{proof}
Suppose, for the sake of contradiction, that $D$ is a defining set for $F_{m,n,t}$ and that there exists a pair $\{a,b\}\subset N(t)$ such that $|S_{a,b}|\geq m+n$. 
Define $G(D)=(V_1\cup V_2,E)$  
(where $V_1=\{r_1,r_2,...,r_m\}$ and $V_2=\{c_1,c_2,...,c_n\}$) to be the bipartite graph that corresponds to the cells of $S_{a,b}$. That is, edge $\{r_i,c_j\}\in E$ if and only if the cell $D_{i,j}$ of $D$ contains neither $a$ nor $b$.  
Since $|S_{a,b}|\geq m+n$, by elementary graph theory, $G(D)$ contains a cycle $C$; necessarily of even length. Partition the edges of $C$ into $M_1$ and $M_2$, each sets of independent edges. 
Starting with $F=F_{m,n,t}$, we construct an $(m,n,t)$-balanced Latin rectangle $F'\neq F$ as follows. For each $\{r_i,c_j\}\in M_1$, replace $a$ with $b$ in cell $F_{i,j}$, so that this cell has two occurrences of $b$ and no occurrences of $a$. 
        For each $\{r_i,c_j\}\in M_2$, replace $b$ with $a$ in cell $F_{i,j}$, so that this cell has two occurrences of $a$ and no occurrences of $b$. 
Since $C$ is a cycle, the resultant $F'$ is still an $(m,n,t)$-balanced Latin rectangle. 
Moreover, $D\subseteq F'\neq F$; thus $D$ is not a defining set, a contradiction.   
\end{proof}

\begin{theorem}
Let $D$ be a defining set of $F_{m,n,t}$ with $t\geq 2$. Then,
\begin{eqnarray*} 
|D| & \geq & mn\left(t-1/2-\sqrt{\frac{t(t-1)(m+n-1)}{2mn}+\frac{1}{4}}\right). \\
& = & mnt\left(1-O\left(\sqrt{\frac{m+n}{mn}}\right)\right).
\end{eqnarray*} 
\end{theorem}

\begin{proof}
Let $D$ be a defining set of $F_{m,n,t}$. 
For each $(i,j)\in N(m)\times N(n)$, let $e_{i,j}$ be the number of symbols in cell $D_{i,j}$. 
Observe that
\begin{eqnarray}
|D| & = & \sum_{i=1}^m \sum_{j=1}^n e_{i,j}
\label{Equ1}
\end{eqnarray}
 and that for each cell, the number of pairs from $N(n)$ which do {\em not} occur in that cell is
$(t-e_{i,j})(t-e_{i,j}-1)/2$. 
However, 
from the previous lemma, the number of cells in which each pair from $N(t)$ does not occur is at most $m+n-1$. 
Thus:
\begin{eqnarray}
\sum_{i=1}^m \sum_{j=1}^n (t-e_{i,j})(t-e_{i,j}-1)/2 & \leq & t(t-1)(m+n-1)/2. \label{eqq2}
\end{eqnarray}

We wish to minimize (\ref{Equ1}) according to this constraint. By the method of Lagrange multipliers, $-|D|$ is maximized when there is a constant $\lambda$ such that
$$-1=\lambda(e_{i,j}-(2t-1)/2)$$ 
 for each $(i,j)\in N(m)\times N(n)$. 
Let $\lambda'=1/\lambda>0$. 
Thus, from Inequality (\ref{eqq2}) above, 
$$mn(\lambda' +1/2)(\lambda'-1/2)\leq t(t-1)(m+n-1)/2.$$
Thus  
$$\lambda'\leq \sqrt{\frac{t(t-1)(m+n-1)}{2mn}+\frac{1}{4}}.$$ 
The result then follows from (\ref{Equ1}). 
\end{proof}


If $m=n=t$ we have a full $n$-Latin square and the following corollary. 

\begin{corollary}
Let $n\geq 2$. If $D$ is a defining set for the full $n$-Latin square, 
$|D|=n^3(1-O(n^{-1/2}))$. 
\label{bbound}
\end{corollary}


\section{Defining sets of full designs}


We start with an elementary result in extremal graph theory. 

\begin{lemma}
Let $G$ be a simple graph with $v$ vertices and  
more than 
$\lfloor (4v-3)/3\rfloor$ 
edges. 
Then $G$ has an even circuit.  
\end{lemma}

\begin{proof}
Suppose $G$ has no even circuits. 
Then for each odd cycle $C$ in $G$, there is no edge in $G$ not belonging to $C$ that is incident with two vertices from $C$. 
 In fact, any two odd cycles in $G$ must be vertex-disjoint. 
Suppose there are $k$ such cycles altogether, where $k\geq 0$.
Observe that any edge not belonging to one of these cycles must be a bridge. 
Thus there are at most $v-1+k\leq v-1 + \lfloor v/3\rfloor$ edges, a contradiction. 
\end{proof}

We first consider the case when block size is equal to $3$. 
In the following lemma, given a subset $D$ of the full design $F(v,3)$, 
for each (unordered) pair $\{a,b\}\subset N(v)$,
we define $s_{a,b}(D)$ to be 
the number of (unordered) pairs $\{i,j\}\subset N(v)\setminus \{a,b\}$ 
such that neither $\{i,j,a\}$ {\em nor} $\{i,j,b\}$ is an element of $D$.

\begin{lemma}
Let $D$ be a subset of the full design $F(v,3)$.
where $v\geq 4$. 
 Suppose there
exist distinct $a,b\in N(v)$ such that
$s_{a,b}(D)>\lfloor (4v-11)/3\rfloor$.  
 Then $D$ is not a defining set for $F(n,3)$. 
\end{lemma}

\begin{proof}
For the sake of contradiction suppose that there exists such a pair $a,b$ and that $D$ is a defining set. 
Construct a simple graph $G$ with vertex set $N(v)\setminus \{a,b\}$ 
and an edge between $i$ and $j$ whenever neither $\{i,j,a\}$ nor $\{i,j,b\}$ in $D$. From the previous lemma, $G$ contains an even circuit $C$; list the edges of $C$ in the form of an Euler circuit $e_1,e_2,\dots e_{2m}$ (where $m\geq 2$), i.e. so that each pair of adjacent edges in the list share one vertex.  

Let $F_1$ and $F_2$ be the set of edges in the list with odd and even subscripts, respectively, so that together they partition the edges of $C$. 
We create a design $F'\neq F$ by adjusting $F=F(v,3)$ as follows. 
For each edge $\{i,j\}\in F_1$,   
replace $\{i,j,a\}\in F$ with $\{i,j,b\}$, so that there are now two occurrences of $\{i,j,b\}$. 
For each edge $\{i,j\}\in F_2$,   
replace $\{i,j,b\}\in F$ with $\{i,j,a\}$, so that there are now two occurrences of $\{i,j,a\}$. 
 
Observe that $F'$ is a design with the same parameter set as $F$ and $D\subseteq F'$. Therefore $D$ is not a defining set, a contradiction. 
\end{proof}

We give an example illustrating the previous lemma. Suppose the following blocks are {\em not} in $D\subset F(v,3)$, for some $v$ such that $\{1,2,3,4,5,a,b\}\subset N(v)$: 
$$\{\{1,2,x\},\{2,3,x\},\{3,1,x\},\{1,4,x\},\{4,5,x\},\{5,1,x\}\mid x\in \{a,b\}\}.$$
Then, there is an even circuit $G$, and in turn sets of edges $F_1$ and $F_2$ as in the proof above:
$$F_1=\{\{1,2\},\{3,1\},\{4,5\}\},\quad F_2=\{\{2,3\},\{1,4\},\{5,1\}\}.$$
Then $D$ is a subset of the following $(v,3,v-2)$-design:
$$F(v,3)\setminus \{\{1,2,a\},\{2,3,b\},\{3,1,a\},\{1,4,b\},\{4,5,a\},\{5,1,b\}\},$$
with the following blocks each occurring twice:
$$\{\{1,2,b\},\{2,3,a\},\{3,1,b\},\{1,4,a\},\{4,5,b\},\{5,1,a\}\}.$$

\begin{theorem}
Let $D$ be a defining set for the full design $F(v,3)$. 
Then 
$$|D|\geq v(v-1)(v-5/2-\sqrt{(32v-85)/12})/6={v \choose 3}(1-O(v^{-1/2})).$$
\label{threee}
\end{theorem}

\begin{proof}
Let $D$ be a defining set of $F(v,3)$.   
For each (unordered) pair $\{i,j\}\subset N(v)$,
let $x_{i,j}$ be the number of triples in $D$ which contain both $i$ and $j$; then $0\leq x_{i,j}\leq v-2$. 
Observe that $3|D|=\sum_{\{i,j\}\subset N(v)} x_{i,j}$. 

From the previous lemma, for each $\{a,b\}\subset N(v)$, $s_{a,b}(D)\leq  (4v-11)/3$. 
Observe that 
$$
\sum_{\{i,j\}\subset N(v)} (v-2-x_{i,j})(v-3-x_{i,j})/2 
=\sum_{\{a,b\}\subset N(v)} s_{a,b}(D)  
\leq {v\choose 2}(4v-11)/3.$$   

We wish to minimize $|D|$ according to this constraint. 
By the method of Lagrange multipliers, $-|D|$ is maximized when there is a constant $\lambda$ such that 
$$-1=\lambda(x_{i,j}-(2v-5)/2)$$
 for each pair $\{i,j\}\subset N(v)$.
Let $\lambda'=1/\lambda >0$.  
Thus, from the inequality above, 
$$(\lambda' +1/2)(\lambda'-1/2)\leq 2(4v-11)/3$$
and $$\lambda' \leq \sqrt{(32v-85)/12}.$$
The result follows. 
\end{proof}

We now obtain a similar result for the full design $F(v,k)$, where $k>3$. 
\begin{lemma}
Let $D$ be a defining set for $F(v,k)$ and let $K\subset N(v)$ be a fixed subset such that $|K|=k-3$.
Let $d_K$ be the number of blocks of $D$ which include $K$ as a subset. Then 
$$d_K\geq (v-k+3)(v-k+2)(v-k+1/2-\sqrt{(32(v-k)+11)/12})/6.$$  
\end{lemma}

\begin{proof}
Without loss of generality, let $N(v)\setminus K=N(v-k+3)$. 
Let $D_K$ be the set of blocks containing the subset $K$. 
Let $B_K$ be the set of blocks formed by deleting $K$ from each block of $D_K$. 
Then $B_K$ is a subset of the full design $F(v-k+3,3)$. 
Suppose that  $B_K$ is not a defining set for $F(v-k+3,3)$; let 
  $F_3'$ be a design with parameters $(v-k+3,3,v-k+1)$ that contains $B_K$ but is not the full design
and let $D_K'$ be the set of blocks obtained by taking the union of $K$ with each block from $F_3'$.
Next, let $F'=(F(v,k)\setminus D_K)\cup D_K'$. 
 Then $F'\neq F(v,k)$, $F'$ has the same parameters as $F(v,k)$ and $D\subset F'$, contradicting the fact that $D$ is a defining set. 
Therefore $B_K$ {\em is} a defining set for $F(v-k+3,3)$. The result follows from Theorem \ref{threee}.  
\end{proof}

\begin{theorem}
Let $D$ be a defining set for the full design $F(v,k)$. 
Then 
$$|D|\geq {v\choose k}\left[1-\frac{\left[1+\sqrt{(32(v-k)+11)/3}\right]}{2(v-k+1)}\right]
=
{v\choose k}(1-O((v-k)^{-1/2})).$$
\label{abigone}
\end{theorem}

\begin{proof}
Let $D$ be a defining set of $F(v,k)$. 
Let $K\subset N(v)$ such that $|K|=k-3$ and let $d_K$ be the number of blocks from $D$ which have $K$ as a subset. 
Then
\begin{eqnarray}
{k\choose k-3}|D|& = &\sum_{K\subset N(v),|K|=k-3} d_K. 
\label{eQone}
\end{eqnarray}
So, from the previous lemma: 
$$|D|\geq \frac{{v\choose k-3}{v-k+3\choose 2}(v-k+1/2- \sqrt{(32(v-k)+11)/12})}{3{k\choose k-3}}$$
and the result follows. 
\end{proof}



\begin{thebibliography}{99}

\bibitem{AMM}
S. Akbari, H.R. Maimani and C.H. Maysoori, Minimal defining sets for full $2$-$(v,3,v-2)$ designs, {\it Australas. J. Combin.} {\bf8} (1993), 55--73.


\bibitem{Cav} N.J. Cavenagh, The theory and application of Latin bitrades: a survey, {\it Math. Slovaca} {\bf 58} (2008), 691--718.


\bibitem{CR2}  N.J. Cavenagh and V. Raass, Critical sets of $2$-balanced Latin rectangles, {\it Ann. Comb.} {\bf20} (2016), 525--538. 

\bibitem{CRa} N.J. Cavenagh and V. Raass, Critical sets of full $n$-Latin squares, {\it Graphs Combin.} {\bf 20} (2016), 1--14. 


\bibitem{DY} F. Demirkale and  E.S. Yazici, On the spectrum of minimal defining sets of full designs, {\it Graphs Combin.} {\bf30} (2014), 141--157. 

\bibitem{DLWY}
D. Donovan, J. Lefevre, M. Waterhouse and E.S. Yazici, Defining sets of full designs with block size three II, {\it Ann. Combin.} {\bf16} (2012), 507--515. 

\bibitem{23}
D. Donovan, J. Lefevre, M. Waterhouse and E.S. Yazici, On defining sets of full designs with block size three, {\it Graphs Combin.} {\bf25} (2009), 825--839.

\bibitem{GS3} K. Gray, A.P. Street, Constructing defining sets of full designs, {\it Util. Math.} {\bf 76} (2008), 91--99.  

\bibitem{GS2} K. Gray, A.P. Street, On defining sets of full designs and of designs related to them, {\it J. Combin. Math. Combin. Comput.} {\bf60} (2007), 97--104. 

\bibitem{GS} K. Gray, A.P. Street, Defining sets, Section IV.13 In: C.J. Colbourn, J.H. Dinitz (eds.) CRC Handbook of Combinatorial Designs, 2nd edn. CRC Press, Boca Raton (2007) pp. 382--385. 

\bibitem{GSS} K. Gray, A.P. Street and R.G. Stanton, Using affine planes to partition full designs with block size three, {\it Ars Combin.} {\bf 97A} (2010), 383--402. 


\bibitem{Ke} A. D. Keedwell, Critical sets in Latin squares and related matters: an update, {\it Util. Math} {\bf 65} 
(2004), 97--131.


\bibitem{KY} E. Kolotoglu and E.S. Yazici, On minimal defining sets of full designs and self-complementary designs, and a new algorithm for finding defining sets of $t$-designs, {\it Graphs Combin.} {\bf 26} (2010), 259--281. 


\bibitem{LW} J. Lefevre and M. Waterhouse, On defining sets of full designs, {\it Discrete Math.} {\bf 310} (2010), 3000-3006.  

\end{thebibliography}
\end{document}